\documentclass[journal,twoside,web]{./ieeecolor}
\usepackage{generic}
\usepackage{cite}
\usepackage{algorithmic}
\usepackage{graphicx}
\usepackage{textcomp}
\usepackage{bm}
\usepackage{ascmac}

\usepackage{amsmath,amssymb, mathrsfs}
\usepackage{bbm}
\usepackage{ascmac}
\usepackage{tikz}
\usepackage{comment}
\usepackage{mathtools}
\usetikzlibrary{calc}
\usetikzlibrary{positioning}
\usetikzlibrary{arrows.meta}

\newtheorem{dfn}{Definition}

\newtheorem{thm}{Theorem}

\newtheorem{prop}{Proposition}

\newtheorem{prob}{Problem}
\newtheorem{rem}{Remark}

\newcommand{\del}{\partial}

\newcommand{\pdiff}[2]{\frac{\del #1}{\del #2}}
\newcommand{\pdiffdiff}[2]{\frac{\del^2 #1}{\del {#2}^2}}
\newcommand{\PAR}[1]{{\left( #1 \right)}}
\newcommand{\PPAR}[1]{{\left\{ #1 \right\}}}
\newcommand{\PPPAR}[1]{{\left[ #1 \right]}}

\newcommand{\R}{\mathbb{R}}
\newcommand{\T}{{\rm T}}

\newcommand{\argmin}{\mathop{\rm arg~min}\limits}

\newcommand{\final}[1]{\textcolor{black}{#1}}

\begin{document}
\title{Structured Hammerstein-Wiener model learning for model predictive control
}
%
%
%
\author{
        Ryuta Moriyasu, Taro Ikeda, Sho Kawaguchi, Kenji~Kashima
\thanks{
©2021 IEEE. Personal use of this material is permitted. Permission from IEEE must be obtained for all other uses, in any current or future media, including reprinting/republishing this material for advertising or promotional purposes, creating new collective works, for resale or redistribution to servers or lists, or reuse of any copyrighted component of this work in other works.
}
\thanks{
R.~Moriyasu and T.~Ikeda are with Toyota Central R\&D Labs, Aichi, Japan. S.~Kawaguchi is with Toyota Industries Coorporation, Aichi, Japan. K.~Kashima is with Graduate School of Informatics,
Kyoto University, Kyoto, Japan.
(e-mail:
{\tt kk@i.kyoto-u.ac.jp}
)
}}
\pagestyle{empty}
\maketitle
\thispagestyle{empty}
\begin{abstract}
This paper aims to improve the reliability of optimal control using models constructed by machine learning methods.
Optimal control problems based on such models are generally non-convex and difficult to solve online.
In this paper, we propose a model that combines the Hammerstein-Wiener model
with input convex neural networks,
which have recently been proposed in the field of machine learning.
An important feature of the proposed model is that resulting optimal control problems are effectively solvable exploiting their convexity and partial linearity while retaining flexible modeling ability.
The practical usefulness of the method is examined through its application to the modeling and control of an engine airpath system.
\end{abstract}

\begin{IEEEkeywords}
Model predictive control, Machine learning, Convex optimization, Input convex neural network
\end{IEEEkeywords}
\section{Introduction}
\label{sec:introduction}
    In recent years, there has been an increase in research on control modeling that utilizes machine learning methods such as neural networks and Gaussian processes for model predictive control (MPC); see e.g., \cite{kocijian01,hedjar01,lenz01,moriyasu01}.
    In the case of complex dynamics, first-principles modeling using physical laws requires advanced knowledge and experience, but there is a possibility that such dynamics can be modeled in a short time without advanced knowledge through data-driven modeling via machine learning methods.
    Also, factors that are difficult to handle and ignored in physical models can be implicitly learned from the data, which can result in more accurate models than physical models in some cases.

	However, machine learning models cannot directly be utilized for numerical optimization-based control methods such as MPC \cite{maciej01,nghiem01}. This is mainly because
	the resulting optimal control problem (OCP) to be solved at each time is a non-convex optimization problem \cite{MPCText01}, whose globally optimal solution is difficult to find. Besides, resulting control laws often have discontinuities that can cause hunting of the control input and lead to reliability issues such as instability of the control system and degradation of the actuator.


    Recently, in the field of machine learning, Input Convex Neural Network (ICNN) \cite{convex01}, which can guarantee the convexity of input-output relations of a model, has been developed, and a control design method that guarantees the convexity of OCPs has been proposed by using Recurrent ICNN (RICNN), which is an extension of ICNN to a recursive structure, as a model for control \final{\cite{convex02,Bunning2020}}.
    A drawback is that the stage cost is limited to monotonically non-decreasing functions with respect to the state variables. This rules out, for instance, a quadratic cost that is a typical choice for regulation and tracking control. In addition, the closed loop property such as stability is not guaranteed. \final{See also Remark \ref{rem:computation} below for a necessity of further reduction of computation burden.}

    On the other hand, in the area of system identification, various nonlinear models with a special structure that makes OCP convex  have been developed. A typical example is the Hammerstein-Wiener (H-W) models, 
    consisting of linear dynamics and static nonlinearity. In the control design using these models, the static nonlinearity can be canceled out
    \cite{fruzzetti1997nolinear,cervantes2003nonlinear}. It means that the OCP with cost allowing state and output regulation becomes a convex problem, which can be efficiently solvable. However, when considering constraints on inputs and outputs, especially for the MIMO cases, it is difficult to guarantee the convexity of the feasible set of variables in the obtained model.
    As a result, the examples of studies that include constraints are mainly for the SISO cases \cite{chan2007model}.
    In addition, for the purpose of control design, it is preferable that the static nonlinear function is \final{bijective}. It is, however, difficult to identify complex nonlinear function while guaranteeing invertibility \cite{PATIKIRIKORALA201249,lawrynczuk2014computationally}.

    The purpose of this paper is to propose a novel model, which we refer to as \emph{structured Hammerstein-Wiener models}, and its control design method by combining the techniques developed in both machine learning and control theory. An important feature is that we can take into account constraints on inputs and outputs in a nonlinear multi-input/output system, and can also ensure global optimality and continuity of the control law.
	This paper is organized as follows: Section \ref{sec:Modeling} details the proposed model and its identification procedure.
	Section \ref{sec:OCP} reveals the properties, e.g.,  uniqueness and continuity, of OCP associated to the proposed model.
	In Section \ref{sec:engine}, the effectiveness of the proposed method is demonstrated with its application to MPC of an engine airpath system.

    \subsubsection*{Notation}
	The set of real and non-negative  numbers are denoted as $\mathbb R$ and ${\mathbb R}_+$. A function is said to be of class $C^r$ if it is $r$ times continuously differentiable. A function $f$ on ${\mathbb R}_+$ is said to be of class $\cal K$ if it is $f(0)=0$ and strictly increasing. For vectors $x,y$, $x_{(i)}$ denotes the $i$-th element of $x$ and $x\odot y$ is the Hadamard product. Function ${\rm softplus}\PAR{v} \coloneqq \log\PAR{1+e^v}>0$. The vector  ${1}_n:=[1,\ldots,1]^\T\in {\mathbb{R}}^n$.

\section{Modeling and Identification}
\label{sec:Modeling}

\subsection{Structured Hammerstein-Wiener Model}

	We denote the input $u\in\mathbb{R}^{n_u}$, the disturbance $d\in\R^{n_d}$. Tracking to the reference and constraints on internal states are described by using the output $y\in\mathbb{R}^{n_y}$ and $z\in\mathbb{R}^{n_z}$, respectively. 	We suppose $n_u=n_y$.
	We assume that the time series data of these signals are available, as well as their time derivatives (true value or difference approximation), if necessary.
	Our goal is to develop systems from
	$u,d$ to $y,z$, and obtain a control law that makes $y$ follow the reference $r\in\mathbb{R}^{n_y}$ while satisfying the box constraints for the input $u=[u_{(1)},\ldots,u_{(n_u)}]^\T$
	\begin{align}\label{eq:constu}
		u\in {\cal U} := \{ u:\underline{u}_{(i)} \leq u_{(i)} \leq \overline{u}_{(i)} \ (i=1,\ldots,n_u) \}
	\end{align}
	and output $z=[z_{(1)},\ldots,z_{(n_z)}]^\T$
	\begin{align}\label{eq:constz}
		z\in {\cal Z} := \{z:z_{(i)} \leq \overline{z}_{(i)} \ (i=1,\ldots,n_z)\}.
	\end{align}

    In this paper, we employ the continuous-time Hammerstein-Wiener model represented by
	\begin{align}
	v(t) &= \Psi(u(t);d(t)),  \label{eq:HWv}\\
	\dot{x}(t) &= A(d(t)) x(t) + B(d(t)) v(t) + c(d(t)), \label{eq:HWx} \\
	y(t) &= \Phi^{-1}(x(t);d(t)), \label{eq:HWy} \\
	z(t) &= \Xi (x(t),v(t);d(t)), \label{eq:HWz}
	\end{align}
	where $v\in\R^{n_u}$ and $x\in\R^{n_y}$ are converted input and output signals.
	We impose the following assumption that makes the resulting OCP effectively solvable.
	\begin{dfn}
	\label{dfn:cHW}
	If functions $\Psi: \R^{n_u} \times \R^{n_d} \rightarrow \R^{n_u}$, $\Phi: \R^{n_y} \times \R^{n_d} \rightarrow \R^{n_y}$, and $\Xi: \R^{n_y} \times \R^{n_u} \times \R^{n_d} \rightarrow \R^{n_z}$ satisfy
        \begin{enumerate}
            \item  $\Psi(\cdot;d)$ and $\Phi(\cdot;d)$ are bijective, and
            \item $\Xi(\cdot,\cdot;d)$ is convex,
        \end{enumerate}
        for any fixed $d\in \R^{n_d}$, the system in \eqref{eq:HWv} to \eqref{eq:HWz} is called a \emph{structured Hammerstein-Wiener model}.
    \end{dfn}

\subsection{Parameterization of bijective and convex mappings}

In this section, we parameterize the bijective and convex mappings in Definition \ref{dfn:cHW}.
In recent years, there have been many proposals in the area of flow-based generative models (Normalizing Flow) \cite{rezende01}. In this area, for the ease of the evaluation of the determinant of the Jacobian, neural networks with a special structure in which the Jacobian is a block triangular matrix are used. Since such property is unnecessary for the purpose of this paper, we extend the Bijective NN \cite{baird2005one}, which is a relatively old known model structure with higher degrees of freedom.

	\begin{dfn}
		Let $\varphi^{(\cdot)}:\R^{n_\xi} \times \R^{n_\eta} \rightarrow \R^{n_\xi}$ be a elementwise nonlinear function. Function $f(\xi;\eta,\theta)$ with $\theta = (\Omega^{(i)},\beta^{(i)})_{i=1}^L$ where $\Omega^{(\cdot)}:\R^{n_\eta} \rightarrow \R^{n_\xi \times n_\xi}$,
		$\beta^{(\cdot)}:\R^{n_\eta} \rightarrow \R^{n_\xi}$ is said to be a bijective neural network (abbr. BNN) if
		\begin{align}
			&f(\xi;\eta,\theta) = \xi^{(L)},\\
			&\xi^{(i)} = \varphi^{(i)} \PAR{\Omega^{(i)}(\eta) \xi^{(i-1)} \!+\! \beta^{(i)}(\eta),\eta} \ (i=1,...,L), \label{BNNAI}\\
			&\xi^{(0)} = \xi,
		\end{align}
		$\Omega^{(i)}(\eta)$ is nonsingular and $\varphi^{(i)}(\cdot,\eta)$ is bijective for any $i,\eta$. If, in addition, $\Omega^{(i)}(\eta)$ is diagonal for any $i,\eta$, $f$ is said to be a diagonal BNN.
		A BNN $f$ is said to be $C^r$-diffeomorphic neural network (abbr.~$C^r$-BNN) if $\varphi^{(i)}, \Omega^{(i)}, \beta^{(i)}\ (i=1,\ldots,L)$ are $C^r$ with respect to $\xi,\eta$.
	\end{dfn}
    As a component of neural networks, $\varphi,\Omega,\beta$ are called as an activation function, weighting matrix, and bias vector, respectively.
    If $f$ is a $C^r$-BNN,  $f(\cdot;\cdot,\theta)$ is a $C^r$ function.
	\begin{prop}\label{BNNAI_bijection}
		Let $f(\xi;\eta,\theta)$ be a BNN. Then, $f(\cdot;\eta,\theta)$ is bijective for any $\eta\in\R^{n_\eta}$.
    \end{prop}
	\begin{proof}
	    The inverse mapping is given by
    	\begin{align}
    		&f^{-1}(\xi;\eta,\theta) = \xi^{(0)},\\
    		&\xi^{(i-1)} = \PAR{\Omega^{(i)}(\eta)}^{\! -1} \PAR{ {\varphi^{(i)}}^{-1} \PAR{ \xi^{(i)} , \eta}\!-\! \beta^{(i)}(\eta)  } \\
    		&\hspace{40pt} (i=1,...,L),\nonumber \\
    		&\xi^{(L)} = \xi.
    	\end{align}
	\end{proof}

	For convex mappings, \cite{convex01} proposed a specific neural network that guarantees the convexity with respect to selected input variables. We utilize similar architecture with modification to make them differentiable.

	\begin{dfn}
	\label{dfn:PICNN}
		Let $\varphi_\zeta^\PAR{\cdot}$ be monotonically non-decreasing convex $C^r$ functions and $\varphi_\eta^\PAR{\cdot}$ be $C^r$ functions. The parameter set $\theta=\{W_\zeta^\PAR{\cdot},W_\xi^\PAR{\cdot},W_{\eta}^\PAR{\cdot},W_{\zeta\eta}^\PAR{\cdot},W_{\xi\eta}^\PAR{\cdot},W_{\eta\eta}^\PAR{\cdot},b_{\eta}^\PAR{\cdot},b_{\zeta\eta}^\PAR{\cdot},b_{\xi\eta}^\PAR{\cdot},b_{\eta\eta}^\PAR{\cdot}\}$, where all elements of  $W_\zeta^\PAR{\cdot}$ are nonnegative. Then,
		$\Xi(y;d,\theta)$ is said to be a $C^r$ partially input convex neural network ($C^r$-PICNN) if
		\begin{align}
		& \Xi \PAR{\xi,\eta;\theta} = \zeta^\PAR{L},\\
		&\zeta^\PAR{i} = \varphi_\zeta^\PAR{i} \left(
		W_\zeta^\PAR{i} \PAR{ \zeta^\PAR{i-1} \odot {\rm softplus}\PAR{v_\zeta^\PAR{i}} } \right.\nonumber\\
		&\ \ \ \ \ \ \ \left. + W_\xi^\PAR{i} \PAR{ \xi \odot v_\xi^\PAR{i}} + v_\eta^{(i)} \right) \ (i\!=\!1,\ldots,L),\\
		& \eta^\PAR{i} = \varphi_\eta^\PAR{i} \PAR{ W_{\eta}^\PAR{i} \eta^\PAR{i-1} \!+\! b_{\eta}^\PAR{i} } \ (i=1,\!\ldots\!,L\!-\!1),\\
			& v_\zeta^\PAR{i} =W_{\zeta\eta}^\PAR{i} \eta^\PAR{i-1} + b_{\zeta\eta}^\PAR{i} \ (i=1,\ldots,L),\\
			& v_\xi^\PAR{i} = W_{\xi\eta}^\PAR{i} \eta^\PAR{i-1} + b_{\xi\eta}^\PAR{i}  \ (i=1,\ldots,L),\\
			& v_\eta^\PAR{i} = W_{\eta\eta}^\PAR{i} \eta^\PAR{i-1} + b_{\eta\eta}^\PAR{i} \ (i=1,\ldots,L),\\
			& \zeta^\PAR{0} = \xi,\ \ \eta^\PAR{0} = \eta.
		\end{align}
	\end{dfn}

By the same argument as in \cite{convex01}, we can show the following:
	\begin{prop}\label{CrPICNN_thm}
		Let $\Xi(\xi,\eta;\theta)$ be a $C^r$-PICNN. Then, $\Xi(\cdot,\eta;\theta)$ is a convex function for any $\theta, \eta$, and $\Xi(\cdot,\cdot;\theta)$ is a class $C^r$ function for any $\theta$.
	\end{prop}

	\subsection{Identification}

    In this section, we discuss the identification procedure under the constraints that $\Psi,\Phi$ are $C^r$-BNN and $\Xi$ is $C^r$-PICNN. Parameterization of such models are given by
	\begin{align}
		\!\!\! v(t)&\!=\!\Psi(u(t),d(t);\theta_\Psi), \label{v}\\
		\!\!\! \dot{x}(t) &\!=\! A(d(t);\theta_A) x \!+\! B(d(t);\theta_B) v \!+\! c(d(t);\theta_c), \label{xdot}\\
		\!\!\! y(t)&\!=\!\Phi^{-1}\!(x(t);d(t),\theta_\Phi), \label{y}\\
		\!\!\! z(t)&\!=\!\Xi(x(t),v(t),d(t);\theta_\Xi).
		\label{z}
	\end{align}
    In most identification methods of the H-W models, linear dynamics and nonlinear function (and its inverse) are determined repeatedly \cite{giri2010block}. On the contrary, we propose a one-shot learning method based on an analytic inverse of the BNN.

	The goal is to determine the parameter $\theta_\bullet$  based on the time series data of $u,d,y,z$. We begin with elimination of $x$ and $v$. From (\ref{y}), we have
	\begin{align}
		x=\Phi(y;d,\theta_\Phi), \label{x}
	\end{align}
	and its time derivative
	\begin{align}
		\dot{x}=\pdiff{\Phi(y;d)}{y} \dot{y} + \pdiff{\Phi(y;d)}{d}\dot{d}. \label{xdot2}
	\end{align}
	Consequently, (\ref{xdot}),(\ref{z}),(\ref{x}),(\ref{xdot2}) leads to
	\begin{align}
		\begin{bmatrix}
			\dot{y}\\
			z
		\end{bmatrix}
		=
		\begin{bmatrix}
		\PAR{\pdiff{\Phi}{y}}^{-1} \PAR{A(d) \Phi + B(d) \Psi + c(d) - \pdiff{\Phi}{d} \dot{d}}\\
		\Xi(\Phi,\Psi,d)
		\end{bmatrix},\label{model}
	\end{align}
	where $\Phi$ and $\Psi$ denote $\Phi(y;d,\theta_\Phi)$ and $\Psi(u;d,\theta_\Psi)$, respectively.
	This means the identification reduces to the following minimization:
	\begin{prob}
		Suppose that time series data $\PPAR{u,y,\dot{y},z,d,\dot{d}}_i (i=1,\ldots,N)$, where $i$ is the label for the data, and a positive definite matrix $K_e$ are given.
		Find
		\begin{align}
			\theta = \PPPAR{\theta_\Psi^\T,\theta_\Phi^\T,\theta_\Xi^\T,\theta_A^\T,\theta_B^\T,\theta_c^\T}^\T \label{params}
		\end{align}
		that minimizes
		\begin{align}
            \frac{1}{N}\sum_{i=1}^{N} e_i^\T K_e e_i,
		\end{align}
		where $e_i$ is the prediction error given by
			\begin{align}
		\begin{bmatrix}
		\dot{y}-\PAR{\pdiff{\Phi}{y}}^{-1} \PAR{A(d) \Phi + B(d) \Psi + c(d) - \pdiff{\Phi}{d} \dot{d}}\\
		z-\Xi(\Phi,\Psi,d)
		\end{bmatrix}.\label{error}
	\end{align}
	\end{prob}
	\noindent
	This problem can be effectively solvable via Stochastic Gradient Decent methods. This procedure can be implemented by standard machine learning libraries such as Tensorflow as far as the Jacobian $\del\Phi/\del y$ is nonsingular.

	\begin{rem}\label{jacob_thm}
		The inverse function theorem tells us that, for any $C^r$-BNN $\Phi$, $\del\Phi/\del y$ is nonsingular; see also the proof of \final{Proposition} \ref{BNNAI_bijection}. Since the set of nonsingular matrices is dense in the matrix field, $\Omega^{(i)}$ included in $\Phi(y;d,\theta_\Phi)$ remains nonsingular during learning almost surely when each element of $\Omega^{(i)}$ is parametrized independently. This can be proven rigorously by introducing a suitable probability space, which will be omitted due to page limitations.
	\end{rem}

\section{Optimization-based Control}
\label{sec:OCP}
\subsection{Uniqueness}

	In this section, we study how to construct an OCP that can guarantee the uniqueness of the solution, aiming at the implementation of online MPC via numerical optimization. Let us discretize the continuous-time system (\ref{eq:HWv})-(\ref{eq:HWz}) with the sampling period $\Delta$ as
	\begin{align}
		v_k&=\Psi(u_k \final{;} d_k), \label{v_d}\\
		x_{k+1} &= A_\Delta (d_k) x_k + B_\Delta (d_k) v_k + c_\Delta(d_k), \label{xdot_d}\\
		y_k&=\Phi^{-1}(x_k \final{;} d_k), \label{y_d}\\
		z_k&=\Xi(x_k,v_k\final{;}d_k),\label{z_d}
	\end{align}
	with $A_\Delta=e^{A\Delta},\allowbreak B_\Delta=(\int_0^\Delta e^{A(\Delta-\tau)} d\tau) B, \allowbreak c_\Delta=(\int_0^\Delta e^{A(\Delta-\tau)} d\tau) c$.
	In what follows, finite-time optimization problem with $k=0,\ldots,n$ is considered, where $n$ represents the prediction horizon length. Denote $U \coloneqq \PPPAR{u_0^\T,\ldots,u_{n-1}^\T}^\T \in \R^{n n_u}$, and $V,X,Z,R$ similarly. $\Psi(V;\bar d):=[\Psi(v_0;\bar d),\ldots,\Psi(v_n;\bar d)]^\T$, and $\Xi(X,V;\bar d)$ similarly. We hereafter assume that $d_i=\bar d, r_i=\bar r,\ i=0,\ldots,n$.
	\begin{prob}\label{ocp2}
		Let $x_0, n$, $\bar d, \bar r$, and $\underline{u}<\overline{u}$ be given. Suppose that $Q\succ 0$, $f_v$ is a convex $C^2$-function, $f_z$ is a convex and monotonically non-decreasing $C^2$-function. Find $U$ that minimizes
		\begin{align}
			& f := E^\T Q E + f_u(U) + f_z(Z) \label{ocp2obj} \\
			& E \coloneqq X - {1}_n \otimes \Phi(\bar r \final{;} \bar d)
		\end{align}
		subject to $u_i \in {\cal U},\ i=0,\ldots,n-1$.
	\end{prob}
	The tracking error $E$ is represented in terms of the internal state $x$. For example, take
	\begin{align}
		\nabla y_0 &\coloneqq \pdiff{\Phi^{-1}(x_0 \final{;} \bar d)}{x},Q_0 \coloneqq (\nabla y_0)^\T(\nabla y_0),\\
		Q &:= I_{n\times n}\otimes Q_0.
		\label{eq:Q}
	\end{align}
    This choice of $Q$ corresponds to $\|Y-R\|^2$ in \eqref{ocp2obj} where $\Phi^{-1}$ is linearly approximated around the initial state $(x_0,\bar d)$.
	Function $f_u$ characterizes the input cost.
	We use $f_z$ to describe a soft constraint for \eqref{eq:constz} in the form of a penalty term, e.g.,
	\begin{align}
		f_z(Z) &= \sum_{i=0}^{n-1} \sum_{j=1}^{n_z}
		\max(0,w_j (z_{ji} - \overline{z}_j)^3), \label{f_z}
	\end{align}
	with weight $w_j$. This is just to avoid that the feasible set is empty and the hard constraint \eqref{eq:constz} can be dealt with similarly.

	\begin{thm}\label{ocp1_convex}
		Suppose that $\Psi$ is a diagonal BNN, $\Phi$ is a BNN, $\Xi$ is PICNN, and $B_\Delta(\bar d_i)$ is nonsingular. Then, Problem \ref{ocp2}
		has a unique stationary solution.
	\end{thm}
	\begin{proof}
	By virtue of the invertibility of $\Psi$, we can regard $V$ as decision variables, instead of $U$. Then, $X$ is affine with respect to $V$:
		\begin{align}
		&X = \bar{A}(\bar d) x_0 + \bar{B}(\bar d) V + \bar{c}(\bar d), \label{eq:affineX}\\
		&\bar{A} \coloneqq
		\begin{bmatrix}
			A_\Delta \\ A_\Delta^2 \\ \vdots \\ A_\Delta^n
		\end{bmatrix},
		\bar{c}\coloneqq
		\begin{bmatrix}
			c_\Delta \\ (A_\Delta+I)c_\Delta \\ \vdots \\ \PAR{\sum_{i=0}^{n-1} A_\Delta^i} c_\Delta
		\end{bmatrix},\\
		&\bar{B}\coloneqq
		\begin{bmatrix}
			B_\Delta & O & \cdots & O \\
			A_\Delta B_\Delta & B_\Delta & \cdots & O \\
			\vdots & & \ddots & \vdots \\
			A_\Delta^{n-1} B_\Delta & A_\Delta^{n-2} B_\Delta & \cdots & B_\Delta
		\end{bmatrix}.
	\end{align}
Also, the convexity of $\Xi$ guarantees that $Z$ is a convex function of $V$. The assumptions on $B_\Delta$ and $Q$ imply  that the Hessian of $E^\T Q E$ is positive-definite. This shows that
\begin{align}
    f(V;x_0,\bar d,\bar r) := E^\T Q E + f_u(\Psi(V;\bar d)) + f_z(\Xi(X,V;\bar d))
\end{align}
is a strictly convex function of $V$. By the assumption \final{that $\Psi$ is a diagonal BNN}, the constraint \eqref{eq:constu} can be represented as a box constraint on $V$,
		for which the feasibility set is not empty.
        Finally, the desired result follows from the theory of nonsmooth optimization \cite{Clarke}.
	\end{proof}

	\begin{rem}\label{rem:computation}
	\final{The} fact that the optimal control problem reduces to a convex problem is attractive because the optimal solution can be obtained via gradient methods. However, some industrial applications require a very fast sampling period (e.g., less than $1$ msec to update control input for application to vehicle engine control discussed in the next section) and low-performance computer implementations. In such cases, convexity alone is not sufficient for online optimization. For example, the evaluation of the gradient of the control cost (via backpropagation) for the model in \cite{convex02} requires the same number of substitutions to the neural network as the horizon length $n$, which can be a bottleneck of the online implementation. On the other hand, concerning the proposed model, $\del f/\del V$ requires the neural network substitution only once since the state transition can be described by a matrix multiplication; see \eqref{eq:affineX}. This property reduces the computation burden significantly.
	\end{rem}

\subsection{Continuity}

    The stationary solution is a function of $(x_0,\bar d,\bar r)$, which can be viewed as a control law.
    We next study the continuity of this control law with respect to $(x_0,\bar d,\bar r)$. Let us introduce Lagrange multiplier $\lambda\in\R^{2n n_u}$ such as
    \begin{align}
		L(V,\lambda;x_0,\bar d,\bar r) \coloneqq f(V;x_0,\bar d,\bar r) + \lambda^{\T} g(V;x_0,\bar d,\bar r).
	\end{align}
	Here, $g\le 0$ represents $u_i \in {\cal U}, i=0,...,n-1$. Because of the convexity of Problem \ref{ocp2}, a $V$ is a globally optimal solution if and only if the KKT condition
		\begin{align}
			\pdiff{L}{V}=0, \ \lambda \odot g=0,\  \lambda \geq 0,\ g \leq 0, \label{kktcond}
		\end{align}
	is satisfied.
	This inequality condition can be rewritten as an equality condition by using Fisher-Burmeister (FB) function \cite{fischer1995ncp} $\phi(a,b) \coloneqq a+b - \sqrt{a^2 + b^2}.$
    That is, the KKT condition (\ref{kktcond}) is equivalent to
		\begin{align}
			F(V,\lambda;x_0,\bar d,\bar r) \coloneqq
			\begin{bmatrix}
			\PAR{\pdiff{L}{V}}^\T \\
				\phi(V,-g)
			\end{bmatrix} = 0.\label{kktsys}
		\end{align}
	Consider the generalized Jacobian\footnote{The generalized Jacobian is needed since $\phi(a,b)$ is not differentiable at $a=b=0$. } $\partial F$ of $F$ \cite{Clarke} with respect to $\tilde{V} \coloneqq \PPPAR{V^\T,\lambda^\T}^\T$.
	Then, we can show the continuity of the optimal control law.
	\begin{thm}\label{controllaw}
		Denote $\tilde{V}^*$ the $(x_0,\bar d,\bar r)$-dependent unique solution to \eqref{kktsys}. Then, the control law defined by
		\begin{align}
			u(x_0,\bar d,\bar r) &\coloneqq \Psi^{-1}( v^*(x,\bar d, \bar r),\bar d ),\label{col1_u} \\
			v^*(x,\bar d,\bar r) &\coloneqq \tilde{V}_{1:n_u}^*,\label{col1_v}
		\end{align}
		is locally Lipschitz continuous,
		where $\tilde{V}_{1:n_u}^*$ represents the first $n_u$ elements of $\tilde{V}^*$, i.e., the optimal input at the first time period.
	\end{thm}
    \begin{proof}
        From the specific structure of Problem \ref{ocp2}, the unique existence of $\tilde V^*$ follows from
        \begin{enumerate}
            \item[P1)] linear independent constraint qualification is satisfied,
            \item[P2)] strong second-order sufficient condition\footnote{$
            \gamma^\T \pdiffdiff{L}{V} \gamma >0$ for any $\gamma\in \R^{2n n_u}$ satisfying $\gamma \neq 0,
            \pdiff{g_i}{V}\gamma=0\ (i\in I_+)$, where $I_+$ is the index set of non-active inequalities. } is satisfied.
        \end{enumerate}
        This also means any $H\in\partial F$ is nonsingular \cite{facchinei1998regularity}. Then, the local Lipschitz continuity of $\tilde V^*$ follows from  the implicit function theorem for locally Lipschitz continuous functions \cite{Clarke}.
    \end{proof}

	As explained in Sections \ref{sec:introduction} and \ref{sec:engine}, the continuity of the control law is significant for implementation. In addition, this continuity guarantees the nonsingularity of $\del F /\del \tilde V$, which is an important property when we apply continuation or homotopy type methods \cite{MPC98}. For example, instead of solving the optimization at each time step, let us update the $\tilde V(t)$ according to $ d\tilde V(t)/dt = -(\del F /\del \tilde V )^{-1} F(\tilde V(t))$. This $\tilde V(t)$ realizes
	$ dF(\tilde V(t))/dt = -F(\tilde V(t))$, and consequently $F(\tilde V(t)) \rightarrow 0$.

\subsection{Control Barrier Function}

In what follows, the reduction of decision variables is discussed for implementation in further computationally severe environments.
Suppose that $\bar d$ and $\bar r$ are stationarily realizable in that there exists $\bar v$ such that  $A(\bar d)\bar x + B(\bar d)\bar v + c(\bar d) = 0$ with $\bar x:= \Phi(\bar r;\bar d)$ and $\bar u:= \Psi^{-1}(\bar v;\bar d)$ satisfies \eqref{eq:constu}. Then, the $v(t)$ that minimizes
\begin{align}
    J = \int_0^\infty (x-\bar x)^\T Q (x-\bar x) + \|v-\bar v\|^2 dt
\end{align}
subject to \eqref{eq:HWx} with $d(t)=\bar d$ is given by
\begin{align}
    v = \bar v + K(x-\bar x),\ K := -B^\T P
\end{align}
where $P$ is the unique positive definite solution to matrix Riccati equation
\begin{align}
    PA + A^\T P - P B^\T B P + Q = 0.
\end{align}
However, $u=\Psi^{-1}(v(t),\bar d)$ does not necessarily satisfy the constraints \eqref{eq:constu} and \eqref{eq:constz}. To guarantee these constraints with lightweight calculation, we propose to use $\Xi(x,v,d)$ as a control barrier function.

\begin{thm}
    Suppose that $\Xi(x,v,d)$ does not depend on $v$.
    Let $\kappa_i(z),i=1,\ldots,n_z,$ be class $\cal K$ functions.
    Define
    \begin{align}
    & {\cal V}_i( x) :=
        \bigg\{ {v}:\Psi^{-1}(v,\bar d) \in {\cal U},  \label{eq:Vi} \\
        & \frac{\partial \Xi_{(i)}(x;\bar d)}{\partial x}
    (A(\bar d) (x-\bar x) + B(\bar d) (v-\bar v) ) \nonumber \\
    & \le \kappa_i(\overline{z}_{(i)}-\Xi_{(i)}(x;\bar d)) \bigg\},\  i=1,\ldots,n_z, \nonumber
    \end{align}
    ${\cal V}(x) := \final{\bigcap}_{i=0}^{n_z} {\cal V}_i(x)$ and
    the control law given by
    \begin{align}
        &{\cal K}(x) := \Psi^{-1}(v^*(x);\bar d) \\
        & v^*(x) := \arg \min_{v\in {\cal V}} \|v-(\bar v + K (x-\bar x) ) \|^2
        \label{eq:filter}
    \end{align}
    If $\Xi(x(0);\bar d)\in \cal Z$, ${\cal V}(x(t))$ is nonempty and $u(t)={\cal K}(x(t))$ is applied, then the constraints \eqref{eq:constu} and \eqref{eq:constz} are satisfied for all $t \ge 0$.
\end{thm}

\begin{proof}
    When $v(t)\in {\cal V}_i$, \begin{align}
        \frac{d \Xi_{(i)}(x(t);\bar d)}{d t} \le
         \kappa_i(\overline{z}_{(i)} - \Xi_{(i)}(x(t);\bar d)).
    \end{align}
    This implies that the time derivative of $\Xi_{(i)}(x(t);\bar d)$ is non-positive
    whenever $\Xi_{(i)}(x(t);\bar d) = \overline{z}_{(i)}$. Therefore, the forward invariance of $\cal Z$ follows from the standard theory for control barrier functions \cite{Ames2019}.
\end{proof}

In order to implement this control law, only \eqref{eq:filter} should be solved online. Note that this is a standard quadratic programming since $\Psi^{-1}(v,\bar d) \in {\cal U}$ is equivalent to a box constraint on $v$ and the inequality constraints included in (\ref{eq:Vi}) are linear inequalities.

\section{Numerical Experiment}
\label{sec:engine}

\subsection{Engine Airpath System}
	\graphicspath{{./figs/}}
	\begin{figure}[b]
		\centering
		\includegraphics[clip, scale=0.6]{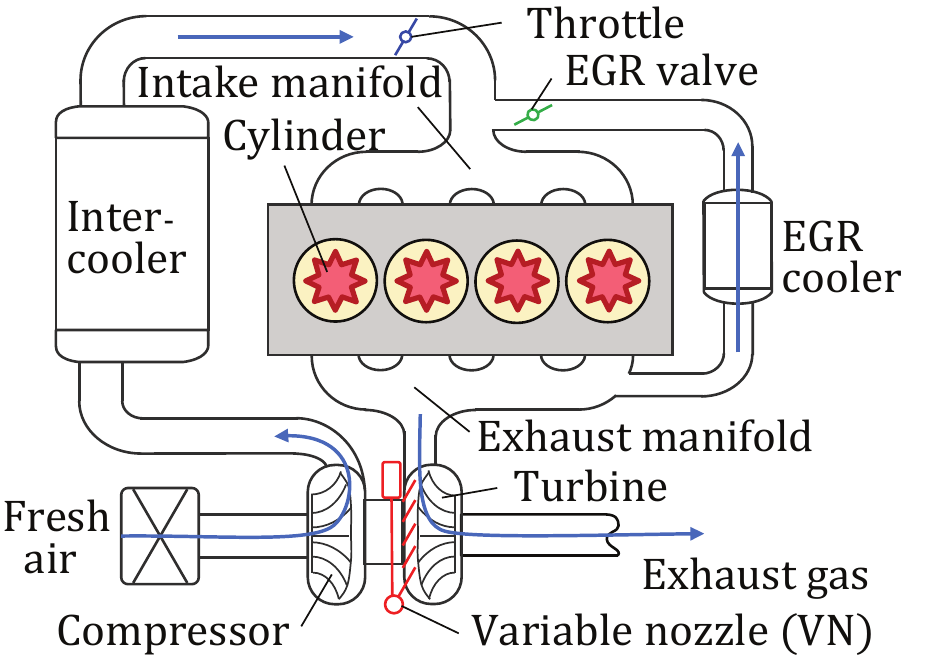}\\
		\caption{Engine airpath system \label{airpath}}
	\end{figure}
	We verify the effectiveness of the proposed method through a numerical experiment on control of engine airpath system depicted in Fig.~\ref{airpath}. This is a heavily nonlinear MIMO system having input- and output-constraints.
    Control input $u=[u_{(1)},u_{(2)},u_{(3)}]^\T$ represents opening positions of variable nozzle, throttle, and Exhaust Gas Recirculation (EGR) valve. Exogenous input $d=[d_{(1)},d_{(2)}]^\T$ is engine speed and fuel injection amount. A tracking reference is given for output $y=[y_{(1)},y_{(2)},y_{(3)}]^\T$, which represents boost pressure (intake manifold gas pressure), EGR ratio, and Pumping Mean Effective Pressure (PMEP).
	A ceiling value is given for the turbine speed $z$.

	By using the data generated by a high-precision simulator equivalent to a real machine, we developed two nonlinear models. One is our proposed model (\ref{v})-(\ref{z}). Concerning the activation function, we employ
	$
		\varphi^{(\cdot)} \PAR{\xi,\eta} = \sinh^{-1} \PAR{ \alpha^{(\cdot)}(\eta) + \sinh \PAR{\xi}},\label{act}
    $
	where $\alpha^{(\cdot)},\beta^{(\cdot)}$ are 3-layered neural networks whose activation function is $\tanh$. For $\Xi$,
	$\varphi_\zeta^\PAR{\cdot}, \varphi_\eta^\PAR{\cdot}$ is  the softplus function, which is a convex and monotonically increasing $C^\infty$ function.
	The other, for comparison purpose, is a standard 3-layered neural network in the form of
	\begin{align}
		&\PPPAR{y_{k+1}^\T,z_k^\T}^\T \nonumber\\
		&= W^{(2)} \tanh\PAR{ W^{(1)} \PPPAR{y_k^\T,u_k^\T,d_k^\T}^\T \!+\! b^{(1)}} + b^{(2)}, \label{nnmodel}
	\end{align}
	where $W^{(\cdot)},b^{(\cdot)}$ is weighting matrices and bias vectors. Detailed architecture such as the number of nodes is tuned so that these models have the almost same degree of representation ability.
	The resulting model accuracy, which is not shown due to page limitation, is comparable.

\subsection{Control System Design}

	The optimization criteria are described for both models. For the proposed model, we solve Problem \ref{ocp2} with $f_u=0$, $f_z$ in \eqref{f_z}, and $Q$ in \eqref{eq:Q}. Then,
	the control law is given by (\ref{col1_u}))\footnote{The proposed control law $u(x,d,r)$ represents $u(\Phi(y,d),d,r)$.}.
	For the 3-layered neural network model, we solve the non-convex optimization problem
	\begin{align}
		u(y,d,r) &\!=\! U^*_{1:n_u}(y,d,r),\\
		U^* (y,d,r) &\!=\! \argmin_{U}\ \ \|E_y\|^2 + f_z(Z) \nonumber\\
		& \ \ \ {\rm s.t. } \ \ \underline{u} \leq\! u_k \!\leq\! \overline{u} \ (k\!=\!0,\ldots,n\!-\!1),\label{ocp3}
	\end{align}
	with $E_y \coloneqq Y - R$ and $f_z$ in \eqref{f_z}. As a reliable nonlinear optimization solver, Sequential Quadratic Programming (SQP) method is applied to both problems with several initial conditions. Note that the convergence of the SQP method to the global optimizer is not guaranteed for non-convex problems.

\subsection{Result}

    Fig.~\ref{ctrllaw} shows the results of the SQP method for obtaining $u$ for the two control laws described above, with only initial $y_{(1)}$ (for SQP) changed and the other values fixed at specific values.
    Fig. 2(a) and (b) show the results using the 3-layered neural network and the proposed method, respectively. The vertical axis is normalized to the upper and lower limits for each element of $u$.
	Thick lines (A) represent the result obtained by SQP with initial input $u$ taken as the middle of upper and lower limits of $u$, while the initial input for
	thin lines (B) is given as the lower limits of $u$.
    The vertical axis is normalized to the upper and lower limits for each element of $u$.

    First, since the OCP for the control law for the 3-layered neural network is a non-convex problem, it can be confirmed that different solutions are obtained in (A) and (B) of Fig.~\ref{ctrllaw}(a).
    In addition, there are discontinuities where the solution bifurcates to different solutions.
	On the other hand, in Fig.~\ref{ctrllaw}(b), we can confirm that the solutions are unique and continuous, indicating the effectiveness of the method.

Finally, the results of the MPC simulation using each control law are shown in Fig.~\ref{mpc}. The dashed and dotted lines in the figure show the results of MPC simulations using each control law.
Observe that the trajectories largely depend on the choice of the initial input.
In particular, the result of (A) shows the fluttering of the input and deviation from the target value, which may be caused by the discontinuity of the control law. On the other hand, the thin solid line is the result of the proposed method, and only one case is shown because the result does not depend on the initial value of the search. In this case, good target tracking and constraint satisfaction are achieved.

	\graphicspath{{./figs/}}
	\begin{figure}[t]
		\centering
		\includegraphics[clip, scale=0.65]{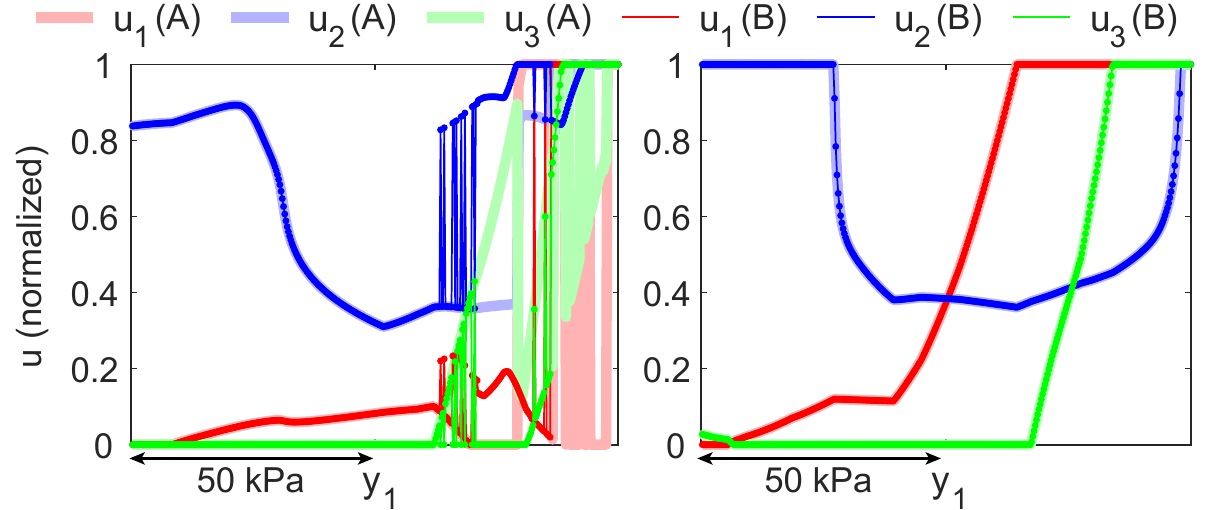}\\
		{\small (a) 3-layered NN based \hspace{30pt} (b) Proposed}
		\caption{Control law \label{ctrllaw}}
		\includegraphics[clip, scale=1.0]{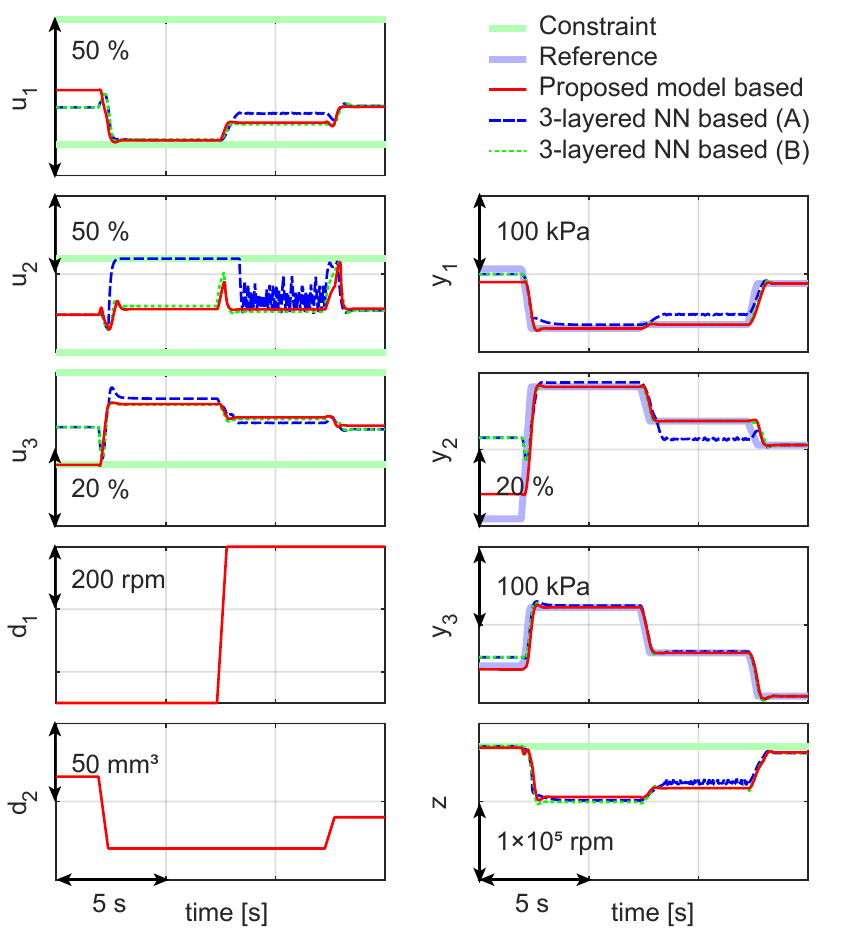}\\
		\caption{Results of model predictive control \label{mpc}}
	\end{figure}

\section{Conclusion}

    In this paper, we proposed a learning model structure that combines the Hammerstein-Wiener model with Bijective NN and Input Convex NN, which are extended to represent disturbance dependency and differentiability. Using this model, we formulated the optimal tracking control problem with input-output constraints as a convex problem and guarantee the continuity of the control law.

    The effectiveness of the method is demonstrated by numerical examples for an engine airpath system, which is a multi-input/output system with input/output constraints. This method is expected to have a wide range of industrial applications, including safety-critical applications because it provides a kind of reliability guarantee for machine learning-based control design.

    The usefulness of the MPC methods described after Theorem 10 has already been confirmed by theory and experiment. The proposed modeling framework can also be extended to differentially flat systems. These results will be presented in a future publication. We are currently working on the relaxation of the diagonality of the input operator $\Psi$ and the application of differentiable MPC.

\newpage
\bibliography{HW_reference}
\bibliographystyle{IEEEtran}

\end{document}